\documentclass[12pt]{article}
\usepackage{mathrsfs}
\usepackage{amsmath,amssymb,amsfonts,latexsym,psfrag,eepic,colordvi}
\usepackage{graphicx, graphics, colordvi, pstricks, pst-node}

\newtheorem{theo}{Theorem}[section]

\newtheorem{Lemma}{Lemma}[section]

\newenvironment {proof} {\noindent{\em Proof.}}{\hspace*{\fill}$\Box$\par\vspace{4mm}}

\allowdisplaybreaks[4]

\begin{document}

\title{Signless laplacian characteristic polynomials of regular graph transformations}

\author{
  Jianping Li$^{\text a, b}$, Bo Zhou$^{\text b}$\footnote{Corresponding author.}\\
$^{\text a}$Faculty of Applied
Mathematics, Guangdong University of Technology,\\
Guangzhou 510090, P. R. China \\
$^{\text b}$Department of Mathematics,
South China Normal University,\\
Guangzhou 510631, P. R. China\\
 Email: {\tt zhoubo@scnu.edu.cn}
 }
\date{}
\maketitle

\begin{abstract}
Let $G$ be a simple $r$-regular graph with $n$ vertices and $m$
vertices.  We give the signless Laplacian
characteristic polynomials of  $xyz$-transformations $G^{xyz}$
of $G$ in terms of $n$, $m$, $r$ and the  signless Laplacian
spectrum of $G$.
\end{abstract}

\section{ Introduction}

We consider simple graphs.
Let $G=(V, E)$
be a graph with vertex set $V(G)$ and edge set $E(G)$.

Let  $V(G)=\{v_1,\dots, v_n\}$. The adjacency matrix $A(G)$ of $G$ is the
$(0, 1)$-matrix $(a_{ij})$ of order $n$ where $a_{ij}=1$ if
$v_iv_j\in E(G)$, and $a_{ij}=0$ otherwise. The degree matrix $D(G)$ of
$G$ is the (diagonal) matrix $(d_{ij})$ of order $n$ where
$d_{ii}$ is the degree of vertex $v_i$ in $G$ and $d_{ij}=0$ for $i\neq j$.
The matrix $L(G)=D(G)-L(G)$ is the Laplacian matrix of $G$, and $Q(G)=D(G)+A(G)$ is
the signless
Laplacian matrix of $G$.

The characteristic polynomials  of $A(G)$, $L(G)$ and $Q(G)$ are called the adjacency, Laplacian and signless Laplacian characteristic polynomials of $G$, respectively. The spectra  of $A(G)$, $L(G)$ and $Q(G)$ are called the adjacency, Laplacian and signless Laplacian spectra of $G$, respectively.

The complement $G^c$ of $G$ is the graph with vertex set
$V(G^c)=V(G)$ and for any $u,v\in V(G)$ and $u\neq v$, $uv\in
E(G^c)$ if and only if $uv\not\in E(G)$ .

Let $G^0$ be the empty graph with $V(G^0)=V(G)$, $G^1$ the complete graph with $V(G^1)=V(G)$, $G^+ = G$,
and $G^-=G^c$.

Let $B(G)$ $(B^c(G))$ be the graph with vertex set $V(G)\cup E(G)$
such that $ve$ is an edge in $B(G)$ (resp., in $B^c (G))$ if and
only if $v\in V(G)$, $e\in E(G)$, and vertex $v$ is incident (resp.,
not incident) to edge $e$ in $G$.

The line graph $G^l$ of $G$ is the graph with vertex set $E(G)$ and two vertices are
adjacent in $G^l$ if and only if the corresponding edges in $G$ are
adjacent.

Let $G$ be a graph and  $x,y,z$ variables in $\{0,1,+,-\}$.
The $xyz$-transformation $G^{xyz}$ of $G$ is the graph with
vertex set $V(G^{xyz})=V(G)\cup E(G)$ and the edge set
$E(G^{xyz})=E(G^x)\cup E((G^l)^y)\cup E(W)$, where $W=B(G)$ if $z=
+$, $W=B^c(G)$ if $z=-$, $W$ is the empty graph with $V(W)=V(G)\cup E(G)$ if $z=0$, and $W$ is the complete bipartite graph with partite sets $V(G)$ and $E(G)$ if $z=1$.

For a regular graph $G$,  the adjacency characteristic polynomials (and the
adjacency spectra) of $G^{00+}$, $G^{+0+}$, $G^{0++}$ and  $G^{+++}$ can be found in
\cite{DM}. The adjacency characteristic
polynomials (and the adjacency spectra) of the other seven $G^{xyz}$
with $x,y,z\in \{+,-\}$ can be found in \cite{JK}. 
Deng et al.  \cite{AAJ} determined
the  Laplacian characteristic polynomials of
$G^{xyz}$ of a regular graph $G$ with $x,y,z\in \{0,1,+,-\}$ (The cases $G^{0++}$, $G^{0+0}$ and $G^{00+}$ have early been given in\cite{AK}).

Let  $f(x,G)=\det(xI_n-Q(G))$ be the signless Laplacian characteristic polynomial of $G$, where $I_n$ is the  the identity matrix of order $n=|V(G)|$.
Now we give  the signless Laplacian characteristic
polynomials of  $xyz$-transformation of an $r$-regular graph
$G$ with $n$ vertices and $m$ edges in terms of $n$,$m$, $r$ and the
signless Laplacian spectrum of $G$ with $x,y,z\in \{0,1,+,-\}$.

\section{Some preliminaries}

Let $G$ be a graph with  $V(G)=\{v_1,\dots, v_n\}$ and  $E(G)=\{e_1,\dots, e_m\}.$
The vertex-edge incidence matrix $R(G)$ of $G$ is the $n\times
m$-matrix $(r_{ij})$, where $r_{ij}=1$ if vertex $v_i$ is incident to
edge $e_j$ and $r_{ij}=0$ otherwise. Then $R(G)R(G)^{\top} =Q(G)$, $R(G)^{\top}R(G)=A(G^l)+2I_m$.
Let $q_1(G), q_2(G), \dots, q_n(G)$ be the  the signless Laplacian
eigenvalues of $G$ arranged in non-decreasing order. Then $q_1\ge 0$ and
$R(G)^{\top}R(G)$  has eigenvalues
 $0$ (of multiplicity $m-n$), $q_1, q_2, \dots, q_n$.

For positive integers $p$ and $q$, let $J_{pq}$ be the all-ones $p\times q$-matrix, and in particular, let $J_p=J_{p,p}$.

In the rest of this paper, $G$ is an $r$-regular graph with $n$
vertices and $m$ edges. Then $2m=rn$. We write $A(G)=A$, $Q(G)=Q$,
$R(G)=R$, and $q_i(G)=q_i$ for $i=1, 2, \dots, n$. In particular,
$q_n=2r$.

%
%

%
%

\begin{Lemma}\label{th3}
Let $P(x,y)$ be a polynomial with two variables and real
coefficients. Then

$(1)$ $P(A,J_n)$ has the eigenvalues $P(r,n)$
and $P(q_i-r, 0)$ for $i = 1, 2,\dots, n-1$, or equivalently,
$P(Q,J_n)$ has the eigenvalues $P(2r,n)$
and $P(q_i, 0)$ for $i = 1, 2,\dots, n-1$, and

$(2)$ $P(R^{\top}R,J_m)$ has the eigenvalues $\sigma_m = P(2r,m)$
and $\sigma_i=P(q_i', 0)$ for $i = 1, 2,\dots, m-1$, where $q_i'=0$
for $i=1,2, \dots, m-n$, and $q_i'=q_{i-m+n}$ for $i=m-n+1, \dots,
m-1$.
\end{Lemma}

\begin{proof}
(1) Let $X_1, X_2,\dots, X_n$ be orthogonal eigenvectors of $A$ such
that $AX_i=(q_i-r)X_i$ for $i=1,2, \dots, n$. Since $A$ has equal row sums,  $X_n = J_{n1}$.
Since $J_{n}^2=nJ_{n}$, $J_nJ_{n1}=nJ_{n1}$ and $AJ_{n1}=rJ_{n1}$,
we have  $A^sJ_n^tJ_{n1}=A^sn^tJ_{n1}=r^sn^tJ_{n1}$ for nonnegative
integers $s$ and $t$, and  thus $P(Q, J_{nn})X_n=P(r, n)X_n$. For
$i=1,2, \dots, n-1$, since $J_nX_i=0$ and $A^sX_i=(q_i-r)^sX_i$, we
have $P(A,J_{n})X_i=P(q_i-r,0)X_i$.

$(2)$ Since $R^{\top}R=A(G^l)+2I_m$, $R(G)^{\top}R(G)$ have equal row sums. Note that
$(R^{\top}R)^sJ_m^tJ_{m1}=(R^{\top}R)^sm^tJ_{m1}=(2r)^sn^tJ_{n1}$
for nonnegative integers $s$ and $t$. By similar argument as in (1), the result follows.
%
%
\end{proof}

\begin{Lemma}\label{l2} Let $B$ and $C$ be square matrices. Then
\[
\left|
\begin{array}{cccc}
B &  E\\
   F   & C
\end{array}
\right|= \left\{
\begin{array}{lll}
|B||C-FB^{-1}E| &\mbox{if $B$ is invertible},\\
  |C||B-EC^{-1}F| &\mbox{if $C$ is invertible}.
\end{array}
\right.
\]
\end{Lemma}

Note that  $q_{n-i}(G^c)=n-2-q_i$ for every $i\in\{1,2,\dots, n-1\}$
and $q_n(G^c)=2(n-r-1)$.  Equivalently, we have the following lemma.

\begin{Lemma}\label{th2}
%
%
$(\lambda-n+2+2r)f(x,G^c)=(-1)^{n}(\lambda-2n+2+2r)f(n-2-\lambda,
G)$.
\end{Lemma}


\section{Signless Laplacian characteristic polynomials of $G^{xyz}$ with $z=0$}

Obviously, $f(\lambda, G^0)=\lambda^n$, $f(\lambda, G^+)=f(\lambda,
G)$ and $f(\lambda, G^1)=(\lambda-2n+2)(\lambda-n+2)^{n-1}$. Note
that  $G^l$ is regular of degree $2r-2$ and $f(G^l,\lambda)=(\lambda
-2r+4)^{m-n}f(\lambda-2r+4, G)$.

For $x, y\in \{0, 1,+, -\}$, $G^{xy0}$ consists of vertex disjoint
$G^x$ and $(G^l)^y)$, and then  $f(\lambda, G^{xy0})=f(\lambda,
G^x)f(\lambda, (G^l)^y)$. Thus we have the following conclusions (16
cases).
\begin{eqnarray*}
f(\lambda,
G^{x00})&=&\lambda^mf(\lambda, G^x) \text{ for } x\in\{0,1,+\},\\
f(\lambda, G^{x10})&=&(\lambda-2m+2)(\lambda-m+2)^{m-1}f(\lambda,
G^x) \text{ for }  x\in\{0,1,+\},\\
 f(\lambda,
G^{-00})&=&(-1)^n\lambda^m(\lambda-n+2+2r)^{-1}\\
&&(\lambda-2n+2+2r)f(n-2-\lambda,
G),\\
f(\lambda,
G^{-10})&=&(-1)^n(\lambda-2m+2)(\lambda-m+2)^{m-1}(\lambda-n+2+2r)^{-1}\\
&&(\lambda-2n+2+2r)f(n-2-\lambda, G)\\
 f(\lambda,
G^{0+0})&=&\lambda^n(\lambda-2r+4)^{m-n}f(\lambda-2r+4,
G^x),\\
f(\lambda,
G^{1+0})&=&(\lambda-2n+2)(\lambda-n+2)^{n-1}(\lambda-2r+4)^{m-n}f(\lambda-2r+4,
G^x),\\
f(\lambda,
G^{++0})&=&(\lambda-2r+4)^{m-n}f(\lambda,G)f(\lambda-2r+4,
G^x),\\
f(\lambda,
G^{-+0})&=&(-1)^n(\lambda-n+2+2r)^{-1}(\lambda-2n+2+2r)\\
&&f(n-2-\lambda,
G)f(\lambda-2r+4, G^x),\\
 f(\lambda,
G^{0-0})&=&(-1)^{n-1}\lambda^n(\lambda-m+4r-2)^{-1}(\lambda+4r-2m-2)\\
&&(\lambda+2r-2-m)^{m-n}f(m-2r-\lambda+2,
G),\\
f(\lambda,
G^{1-0})&=&(-1)^{n-1}(\lambda-2n+2)(\lambda-n+2)^{n-1}(\lambda-m+4r-2)^{-1}\\
&&(\lambda+4r-2m-2)
(\lambda+2r-2-m)^{m-n}f(m-2r-\lambda+2,
G), \\
f(\lambda,
G^{+-0})&=&(-1)^{n-1}(\lambda-m+4r-2)^{-1}(\lambda+4r-2m-2)(\lambda+2r-2-m)^{m-n}\\&&f(\lambda,G)f(m-2r-\lambda+2,
G),\\
f(\lambda, G^{--0})&=&-(\lambda-n+2+2r)^{-1}(\lambda-2n+2+2r)
(\lambda-m+4r-2)^{-1}\\
&&(\lambda+4r-2m-2)(\lambda+2r-2-m)^{m-n}\\
&&f(n-2-\lambda,
G)f(m-2r-\lambda+2, G).
\end{eqnarray*}

\section{Signless Laplacian characteristic polynomials of $G^{xyz}$ with  $z=1$}

Note that $G^{001}$
 is a complete bipartite graph and $G^{111}=K_{n+m}$. Then
 \[ f(\lambda,
G^{001})=\lambda(\lambda-m-n)(\lambda-m)^{n-1}(\lambda-n)^{m-1},\]
\[f(\lambda, G^{111})=(\lambda-2n-2m+2)(\lambda-m-n+2)^{m+n-1}.\]

\begin{theo}
 \begin{eqnarray*}
 f(\lambda,G^{-01})&=&[(\lambda-n)(\lambda-2n-m+2r+2)-mn](\lambda-n)^{m-1}\\
 &&\prod\limits_{i=1}^{n-1}(\lambda-n-m+2+q_i),
\end{eqnarray*}
\end{theo}

\begin{proof} Obviously,
\begin{eqnarray*}
A(G^{-01})= \left(
\begin{array}{cccc}
A(G^c)  &  J_{nm}\\
J_{mn} & 0
\end{array}
\right)=\left(
\begin{array}{cccc}
J_n-I_n-A  &  J_{nm}\\
J_{mn} & 0
\end{array}
\right)\end{eqnarray*}   and  \begin{eqnarray*} D(G^{-01})= \left(
\begin{array}{cccc}
(n+m-r-1)I_n  &  0\\
0 & nI_m
\end{array}
\right).\end{eqnarray*}
Then
\begin{eqnarray*}
\begin{array}{lll}
f(\lambda, G^{-01})&=& \det(\lambda I_{n+m}-Q(G^{-01}))\\
&=&\left|
\begin{array}{cccc}
(\lambda-n-m+r+2)I_n-J_n+A  &  -J_{nm}\\
     -J_{mn}      & (\lambda-n)I_m
\end{array}
\right|.
\end{array}
\end{eqnarray*}
Clearly, it is sufficient to prove our claim for $\lambda\neq n$.
By Lemma \ref{l2},
\[
f(\lambda, G^{-01})=(\lambda-n)^{m-n}|B|.
\]
where
\[
B=(\lambda-n)(\lambda-n-m+r+2)I_n-(\lambda-n)J_n+(\lambda-n)A-mJ_n.
\]
By Lemma \ref{th3}, the
eigenvalues of $B$ are
\begin{eqnarray*}
\sigma_n
&=&(\lambda-n)(\lambda-n-m+r+2)-(\lambda-n)n+r(\lambda-n)-mn
\\&=&(\lambda-n)(\lambda-2n-m+2r+2)-mn
\end{eqnarray*}
and
\begin{eqnarray*}
\sigma_i&=&(\lambda-n)(\lambda-n-m+r+2)+(\lambda-n)(q_i-r)\\
&=&(\lambda-n)(\lambda-n-m+2+q_i)
\end{eqnarray*}
for $i=1,2, \dots, n-1$.
Then
\[
|B|=[(\lambda-n)(\lambda-2n-m+2r+2)-mn](\lambda-n)^{n-1}\prod\limits_{i=1}^{n-1}(\lambda-n-m+2+q_i),
 \]
 and thus the result follows.
\end{proof}

Similarly, we have the following theorem.

\begin{theo}
\begin{eqnarray*}
f(\lambda,
G^{101})&=&[(\lambda-n)(\lambda-m-2n+2)-mn](\lambda-m-n+2)^{n-1}(\lambda-n)^{m-1},\\
f(\lambda,
G^{+01})&=&[(\lambda-n)(\lambda-2r-m)-mn](\lambda-n)^{m-1}\prod\limits_{i=1}^{n-1}(\lambda-m-q_i).
\end{eqnarray*}
\end{theo}

\begin{theo}
\[ f(\lambda,
G^{+11})=[(\lambda-m-2r)(\lambda-2m-n+2)-mn](\lambda-m-n+2)^{m-1}\prod\limits_{i=1}^{n-1}(\lambda-m-q_i).
\]
\end{theo}

\begin{proof} Obviously,
\begin{eqnarray*}
A(G^{+11})= \left(
\begin{array}{cccc}
A  &  J_{nm}\\
J_{mn} & J_m-I_m
\end{array}
\right)
\end{eqnarray*}
and
\begin{eqnarray*}
D(G^{+11})= \left(
\begin{array}{cccc}
(m+r)I_n  &  0\\
0 & (m+n-1)I_m
\end{array}
\right).\end{eqnarray*}
Then
\begin{eqnarray*}
f(\lambda, G^{+11})&=&\det(\lambda I_{n+m}-Q(G^{+11}))\\
&=& \left|
\begin{array}{cccc}
(\lambda-m-r)I_n-A  &  -J_{nm}\\
     -J_{mn}      & (\lambda-m-n+2)I_m-J_m
\end{array}
\right|.
\end{eqnarray*}
Let
\begin{eqnarray*}
M&=& \left(
\begin{array}{cccc}
(\lambda-m-r)I_n-A  &  -J_{nm}\\
     -J_{mn}      & (\lambda-m-n+2)I_m-J_m
\end{array}
\right).
\end{eqnarray*}
Obviously, $R^{\top}J_{nm}=2J_m$. Thus
multiplying the first row of the block matrix $M$ by
$-\frac{1}{2}R^{\top}$ and adding the result to the second row of $M$, we
obtain a new matrix
\begin{eqnarray*}
M_1=\left(
\begin{array}{cccc}
(\lambda-m-r)I_n-A  &  -J_{nm}\\
     \frac{m+r-\lambda}{2}R^{\top}+\frac{1}{2}R^{\top}A-J_{mn}      &
     (\lambda-m-n+2)I_m
\end{array}
\right).
\end{eqnarray*} Clearly, $f(\lambda, G^{+11})=|M|=|M_1|$ and it is
sufficient to prove our claim for $\lambda\neq m+n-2$. By
Lemma \ref{l2},
\[
f(\lambda, G^{+11})=(\lambda-m-n+2)^{m-n}|B|,
\]
where
\begin{eqnarray*}
B&=&(\lambda-m-r)(\lambda-n-m+2)I_n-(\lambda-m-n+2)A+\frac{m+r-\lambda}{2}J_{nm}R^{\top}\\
&&+\frac{1}{2}J_{nm}R^{\top}A-J_{nm}J_{mn}.
\end{eqnarray*}
Since $J_{nm}R^{\top}=rJ_n$, $J_{nm}R^{\top}A=r^2J_n$ and
$J_{nm}J_{mn}=mJ_n$, we have
\[B=(\lambda-m-r)(\lambda-n-m+2)I_n-(\lambda-m-n+2)A+\frac{(m+2r-\lambda)r}{2}J_{n}-mJ_{n}.\]
By Lemma \ref{th3}, the
eigenvalues of $B$ are
\begin{eqnarray*}
\sigma_n
&=&(\lambda-m-r)(\lambda-n-m+2)-(\lambda-m-n+2)r+\frac{(m+2r-\lambda)r}{2}n-mn
\\&=&(\lambda-m-2r)(\lambda-n-2m+2)-mn
\end{eqnarray*} and
\begin{eqnarray*}
\sigma_i&=&(\lambda-m-r)(\lambda-n-m+2)-(\lambda-m-n+2)(q_i-r)\\
&=&(\lambda-n-m+2)(\lambda-m-q_i).
\end{eqnarray*}
Then $|B|=[(\lambda-m-2r)(\lambda-n-2m+2)-mn](\lambda-n-m+2)^{n-1}\prod\limits_{i=1}^{n-1}(\lambda-m-q_i)$,  and thus the result follows.
\end{proof}

Similarly, we have the following theorem.

\begin{theo}
\begin{eqnarray*}
f(\lambda,
G^{011})&=&[(\lambda-m)(\lambda-2m-n+2)-mn](\lambda-m)^{n-1}(\lambda-m-n+2)^{m-1},\\
f(\lambda,
G^{-11})&=&[(\lambda-2m-n+2)(\lambda-m-2n+2r+2)-mn](\lambda-m-n+2)^{m-1}\\&&\prod\limits_{i=1}^{n-1}(\lambda-m-n+2+q_i).
\end{eqnarray*}
\end{theo}

\begin{theo}
\begin{eqnarray*}
f(\lambda,
G^{0+1})&=&[(\lambda-m)(\lambda-n-4r+4)-mn](\lambda-n-2r+4)^{m-n}(\lambda-m)^{n-1}
\\&&
\prod\limits_{i=1}^{n-1}(\lambda-n-2r+4-q_i).
\end{eqnarray*}
\end{theo}
\begin{proof} Obviously,
\begin{eqnarray*}
A(G^{0+1})= \left(
\begin{array}{cccc}
0  &  J_{nm}\\
J_{mn} & A(G^l)
\end{array}
\right)=\left(
\begin{array}{cccc}
0  &  J_{nm}\\
J_{mn} & R^{\top}R-2I_m
\end{array}
\right)\end{eqnarray*}   and   \[ D(G^{0+1})= \left(
\begin{array}{cccc}
mI_n  &  0\\
0 & (n+2r-2)I_m
\end{array}
\right).
\]
Then
\begin{eqnarray*}
f(\lambda, G^{0+1})&=& \left|
\begin{array}{cccc}
(\lambda-m)I_n  &  -J_{nm}\\
     -J_{mn}      & (\lambda-n-2r+4)I_m-R^{\top}R
\end{array}
\right|.
\end{eqnarray*}
Clearly, it is sufficient to prove our claim for $\lambda\neq m$. By
Lemma  \ref{l2} (and the fact that $J_{mn}J_{nm}=nJ_m$),
 $f(\lambda, G^{0+1})=(\lambda-m)^{n-m}|B|, $ where
\[ B=(\lambda-m)(\lambda-n-2r+4)I_m-(\lambda-m)R^{\top}R-nJ_m.\]
By Lemma \ref{th3}, the eigenvalues of $B$ are
\begin{eqnarray*}
\sigma_m&=&(\lambda-m)(\lambda-n-2r+4)-2r(\lambda-m)-nm
\\&=&(\lambda-m)(\lambda-n-4r+4)-nm,
\end{eqnarray*}
 for $1\leq j\leq m-n$, \[ \sigma_j=(\lambda-m)(\lambda-n-2r+4),\]
 and for $m-n+1\leq j\leq m-1$
\begin{eqnarray*}
\sigma_i&=&(\lambda-m)(\lambda-n-2r+4)-(\lambda-m)q_j'\\
&=&(\lambda-m)(\lambda-n-2r+4)-(\lambda-m)q_{j-m+n}\\
&=&(\lambda-m)(\lambda-n-2r+4-q_{j-m+n}).
\end{eqnarray*}
 Then
\begin{eqnarray*}
|B|&=&[(\lambda-m)(\lambda-n-4r+4)-mn](\lambda-n-2r+4)^{m-n}(\lambda-m)^{m-1}
\\&&
\prod\limits_{i=1}^{n-1}(\lambda-n-2r+4-q_i),
\end{eqnarray*}
 and thus the result follows.
\end{proof}

Similarly, we have the following theorem.

\begin{theo}
\begin{eqnarray*}
f(\lambda,
G^{0-1})&=&[(\lambda-m)(\lambda-2m-n+4r-2)-mn](\lambda-m-n+2r-2)^{m-n}
\\ &&(\lambda-m)^{n-1}
\prod\limits_{i=1}^{n-1}(\lambda-m-n-2+2r+q_i).
\end{eqnarray*}
\end{theo}

\begin{theo}
\begin{eqnarray*}
 f(\lambda,
G^{1+1})&=&[(\lambda-2n-m+2)(\lambda-n-4r+4)-mn](\lambda-n-2r+4)^{m-n}\\&&(\lambda-n-m+2)^{n-1}
\prod\limits_{i=1}^{n-1}(\lambda-n-2r+4-q_i).
 \end{eqnarray*}
\end{theo}
\begin{proof}  
Obviously,
\begin{eqnarray*}
A(G^{1+1})= \left(
\begin{array}{cccc}
J_n-I_n  &  J_{nm}\\
J_{mn} & R^{\top}R-2I_m
\end{array}
\right)
\end{eqnarray*}
and
\begin{eqnarray*}
 D(G^{1+1})= \left(
\begin{array}{cccc}
(m+n-1)I_n  &  0\\
0 & (n+2r-2)I_m
\end{array}
\right).\end{eqnarray*}
Then
\begin{eqnarray*}
f(\lambda, G^{1+1})&=& \left|
\begin{array}{cccc}
(\lambda-m-n+2)I_n-J_n  &  -J_{nm}\\
     -J_{mn}      & (\lambda-n-2r+4)I_m-R^{\top}R
\end{array}
\right|.
\end{eqnarray*}
Clearly, it is sufficient to prove our claim for $\lambda\neq
m+2n-2$. Let
\begin{eqnarray*}
M=\left(
\begin{array}{cccc}
(\lambda-m-n+2)I_n-J_n  &  -J_{nm}\\
     -J_{mn}      & (\lambda-n-2r+4)I_m-R^{\top}R
\end{array}
\right).
\end{eqnarray*}
Multiplying the first row of the block
matrix $M$ by $\frac{1}{\lambda-2n-m+2}J_{mn}$ and adding the result
to the second row of $M$, we obtain a new matrix
 \begin{eqnarray*}
M_1= \left(
\begin{array}{cccc}
(\lambda-m-n+2)I_n-J_n  &  -J_{nm}\\
    0     & \frac{-n}{\lambda-2n-m+2}J_m+(\lambda-n-2r+4)I_m-R^{\top}R
\end{array}
\right), \end{eqnarray*}
Then
\begin{eqnarray*}
f(\lambda, G^{+11})&=&|M_1|\\
&=&|(\lambda-m-n+2)I_n-J_n|\\
&&\cdot
\left|(\lambda-n-2r+4)I_m-R^{\top}R+\frac{-n}{\lambda-2n-m+2}J_m\right|
\\&=&(\lambda-n-m+2)^{n-1}(\lambda-2n-m+2)^{1-m}|B|,
\end{eqnarray*} where
\[B=(\lambda-2n-m+2)(\lambda-n-2r+4)I_m-(\lambda-2n-m+2)R^{\top}R-nJ_m.\]
By Lemma \ref{th3}, the
eigenvalues of $B$ are
\begin{eqnarray*}
\sigma_m &=&(\lambda-2n-m+2)(\lambda-n-2r+4)-(\lambda-2n-m+2)\cdot
2r-mn
\\&=&(\lambda-2n-m+2)(\lambda-n-4r+4)-mn
\end{eqnarray*}
for $1\leq j\leq m-n$,
\begin{eqnarray*}
\sigma_j=(\lambda-2n-m+2)(\lambda-n-2r+4)
\end{eqnarray*}
 and for $m-n+1\leq j\leq m-1$,
\begin{eqnarray*}
\sigma_j&=&(\lambda-2n-m+2)(\lambda-n-2r+4)-(\lambda-2n-m+2)q_j'\\
&=& (\lambda-2n-m+2)(\lambda-n-2r+4)-(\lambda-2n-m+2)q_{j-m+n}
\\&=&(\lambda-2n-m+2)(\lambda-n-2r+4-q_{j-m+n}).
\end{eqnarray*}
 Then
\begin{eqnarray*}
|B|&=&[(\lambda-2n-m+2)(\lambda-n-4r+4)-mn](\lambda-n-2r+4)^{m-n}\\&&(\lambda-2n-m+2)^{m-1}
\prod\limits_{i=1}^{n-1}(\lambda-n-2r+4-q_i),
\end{eqnarray*}
 and thus the result follows.
\end{proof}

Similarly, we have the following theorem.

\begin{theo} 
\begin{eqnarray*}
f(\lambda,
G^{++1})&=&[(\lambda-2r-m)(\lambda-n-4r+4)-mn](\lambda-n-2r+4)^{m-n}\\&&
\prod\limits_{i=1}^{n-1}(\lambda-n-2r+4-q_i)(\lambda-m-q_i),\\
 f(\lambda,
G^{-+1})&=&[(\lambda-2n-m+2r+2)(\lambda-n-4r+4)-mn]\\
&&(\lambda-n-2r+4)^{m-n}\\&&
\prod\limits_{i=1}^{n-1}(\lambda-n-2r+4-q_i)(\lambda-n-m+2+q_i),\\
f(\lambda,
G^{1-1})&=&[(\lambda-2n-m+2)(\lambda-2m-n+4r-2)-mn]\\
&&(\lambda-m-n-2+2r)^{m-n}\\&&(\lambda-n-m+2)^{n-1}
\prod\limits_{i=1}^{n-1}(\lambda-m-n-2+2r+q_i),\\
f(\lambda,
G^{+-1})&=&[(\lambda-2r-m)(\lambda-2m-n+4r-2)-mn]\\
&&(\lambda-m-n-2+2r)^{m-n}\\&&
\prod\limits_{i=1}^{n-1}(\lambda-m-q_i)(\lambda-m-n-2+2r+q_i),\\
f(\lambda,
G^{--1})&&=[(\lambda-2n-m+2r+2)(\lambda-2m-n+4r-2)-mn]\\&&(\lambda-m-n+2r-2)^{m-n}\\
&&\prod\limits_{i=1}^{n-1}(\lambda-m-n+2+q_i)(\lambda-n-m+2r-2+q_i).
\end{eqnarray*}
\end{theo}

\section{Signless Laplacian characteristic polynomials of $G^{xyz}$ with $z=+$}

 Since $G^{00+}$
 is a bipartite graph, the signless Laplacian eigenvalues of  $G^{00+}$ are the
 same as the Laplacian eigenvalues of $G^{00+}$. Thus
 \[f(\lambda,
G^{00+})=\lambda(\lambda-r-2)(\lambda-2)^{m-n}\prod
\limits_{i=1}^{n-1}[(\lambda-2)(\lambda-r)-q_i].\]

\begin{theo}
\[ f(\lambda,
G^{10+})=[\lambda^2-(r+2n)\lambda+4n-4](\lambda-2)^{m-n}\prod\limits_{i=1}^{n-1}[(\lambda-r-n+2)(\lambda-2)-q_i].\]
\end{theo}
\begin{proof} Obviously,
\[
A(G^{10+})= \left(
\begin{array}{cccc}
J_n-I_n  &  R\\
R^{\top} & 0
\end{array}
\right)  \quad and  \quad  D(G^{10+})= \left(
\begin{array}{cccc}
(r+n-1)I_n  &  0\\
0 & 2I_m
\end{array}
\right).
\]
Then
\[
f(\lambda, G^{10+})= \left|
\begin{array}{cccc}
(\lambda-r-n+2)I_n-J_n  &  -R\\
    -R^{\top}      & (\lambda-2)I_m
\end{array}
\right|.
\]
Clearly, it is sufficient to prove our claim for $\lambda\neq 2$.
By Lemma \ref{l2} (and the fact $RR^{\top}=Q$),
\begin{eqnarray*}
f(\lambda, G^{10+})&=& |(\lambda-2)I_m|\cdot
\left|(\lambda-r-n+2)I_n-J_n-\frac{1}{\lambda-2}RI_mR^{\top}\right|\\
&=&(\lambda-2)^{m-n}\cdot
|(\lambda-r-n+2)(\lambda-2)I_n-(\lambda-2)J_n-Q|.
\end{eqnarray*}
Let $B=(\lambda-r-n+2)(\lambda-2)I_n-(\lambda-2)J_n-Q$.  By Lemma
\ref{th3}, the eigenvalues of $B$ are
\[\sigma_n=(\lambda-r-n+2)(\lambda-2)-(\lambda-2)n-2r=\lambda^2-(r+2n)\lambda+4n-4\]
and for $1\leq i\leq n-1$,
\[\sigma_i=(\lambda-r-n+2)(\lambda-2)-q_i.\]
Then
\[
|B|=[\lambda^2-(r+2n)\lambda+4n-4]\prod\limits_{i=1}^{n-1}[(\lambda-r-n+2)(\lambda-2)-q_i],\]
 and thus the result follows.
\end{proof}

Similarly, we have the following theorem.

\begin{theo} 
\begin{eqnarray*}
f(\lambda,
G^{+0+})&=&[\lambda^2-(2+3r)\lambda+4r](\lambda-2)^{m-n}\prod\limits_{i=1}^{n-1}[(\lambda-2)(\lambda-r-q_i)-q_i],\\
f(\lambda,
G^{-0+})&=&[(\lambda-2)(\lambda-2n+r+2)-2r](\lambda-2)^{m-n}\\
&&\prod\limits_{i=1}^{n-1}[(\lambda-2)(\lambda-n-r+2+q_i)-q_i].
\end{eqnarray*}
\end{theo}

\begin{theo}
\[  f(\lambda,
G^{01+})=[(\lambda-r)(\lambda-2m)-2r](\lambda-m)^{m-n}\prod\limits_{i=1}^{n-1}[(\lambda-r)(\lambda-m)-q_i].\]
\end{theo}

\begin{proof} Obviously,
\[
A(G^{01+})= \left(
\begin{array}{cccc}
0  &  R\\
R^{\top} & J_m-I_m
\end{array}
\right)\quad and   \quad   D(G^{01+})= \left(
\begin{array}{cccc}
rI_n  &  0\\
0 & (m+1)I_m
\end{array}
\right).
\]
Then
\[
\begin{array}{lll}
f(\lambda, G^{01+})&=& \left|
\begin{array}{cccc}
(\lambda-r)I_n &  -R\\
     -R^{\top}      & (\lambda-m)I_m-J_m
\end{array}
\right|.
\end{array}
\]
Let \[ M=\left(
\begin{array}{cccc}
(\lambda-r)I_n &  -R\\
     -R^{\top}      & (\lambda-m)I_m-J_m
\end{array}
\right).
\]
 Multiplying the first row of the block matrix $M$ by
$-\frac{1}{2}J_{mn}$ and adding the result to the second row of $M$,
we obtain a new matrix
 \[
M_1=\left(
\begin{array}{cccc}
(\lambda-r)I_n &  -R\\
\frac{r-\lambda}{2}J_{mn}-R^{\top}     & (\lambda-m)I_m
\end{array}
\right).
\]
Obviously, $f(\lambda, G^{01+})=|M|=|M_1|$, and it is sufficient
to prove our claim for $\lambda\neq m$. By Lemma \ref{l2},
\begin{eqnarray*}
f(\lambda, G^{01+})=(\lambda-m)^{m-n}
\left|(\lambda-r)(\lambda-m)I_n+\frac{r-\lambda}{2}RJ_{mn}-RR^{\top}\right|.
\end{eqnarray*}
Let
$B=(\lambda-r)(\lambda-m)I_n+\frac{r-\lambda}{2}RJ_{mn}-RR^{\top}$.
Obviously, $RJ_{mn}=rJ_n$. Thus
\[
B=(\lambda-r)(\lambda-m)I_n+\frac{r-\lambda}{2}rJ_n-Q.\]
By Lemma \ref{th3} (and the fact that $rn=2m$),
the eigenvalues of $B$ are
\[\sigma_n=(\lambda-r)(\lambda-m)+\frac{r-\lambda}{2}rn-2r=(\lambda-r)(\lambda-2m)-2r\]
and for  $1\leq i\leq n-1$,
\[
\sigma_i=(\lambda-r)(\lambda-m)-q_i.
\]
Then \[  |B|=[(\lambda-r)(\lambda-2m)-2r]\prod\limits_{i=1}^{n-1}[(\lambda-r)(\lambda-m)-q_i],\]
 and thus the result follows.
\end{proof}

Similarly, we have the following theorem.

\begin{theo}
\begin{eqnarray*}
 f(\lambda,
G^{11+})&=&[(\lambda-r-2n+2)(\lambda-2m)-2r](\lambda-m)^{m-n}\\
&&\prod\limits_{i=1}^{n-1}[(\lambda-r-n+2)(\lambda-m)-q_i],\\
  f(\lambda,
G^{-1+})&=&[(\lambda-2m)(\lambda-2n+r+2)-2r](\lambda-m)^{m-n}\\
&&\prod\limits_{i=1}^{n-1}[(\lambda-m)(\lambda-n+r+2-q_i)-q_i],\\
 f(\lambda,
G^{+1+})&=&[(\lambda-2m)(\lambda-3r)-2r](\lambda-m)^{m-n}\\
&&\prod\limits_{i=1}^{n-1}[(\lambda-m)(\lambda-r-q_i)-q_i].
\end{eqnarray*}
\end{theo}

\begin{theo}\label{th4}
\begin{eqnarray*}
f(\lambda,
G^{+++})&=&[(\lambda-3r+2)(\lambda-4r)](\lambda-2r+2)^{m-n}\\
&&\prod\limits_{i=1}^{n-1}[(\lambda-r-q_i)(\lambda-2r+2-q_i)-q_i].
\end{eqnarray*}
\end{theo}

\begin{proof} Obviously,
\[
A(G^{+++})= \left(
\begin{array}{cccc}
A  &  R\\
R^{\top} & A(G^l)
\end{array}
\right)\quad and   \quad   D(G^{+++})= \left(
\begin{array}{cccc}
2rI_n  &  0\\
0 & 2rI_m
\end{array}
\right).
\]
Then
\[
\begin{array}{lll}
f(\lambda, G^{+++})&=& \left|
\begin{array}{cccc}
(\lambda-2r)I_n-A  & -R\\
     -R^{\top}      & (\lambda-2r+2)I_m-R^{\top}R
\end{array}
\right|.
\end{array}
\]
Let \[ M=\left(
\begin{array}{cccc}
(\lambda-2r)I_n-A  & -R\\
     -R^{\top}      & (\lambda-2r+2)I_m-R^{\top}R
\end{array}
\right).
\]
 Multiplying the first row of the block matrix $M$ by
$-R^{\top}$ and adding the result to the second row of $M$, we obtain a
new matrix
 \[
\begin{array}{lll}
M_1&=& \left(
\begin{array}{cccc}
(\lambda-2r)I_n -A &  -R\\
(2r-\lambda-1)R^{\top}+R^{\top}A      & (\lambda-2r+2)I_m
\end{array}
\right).
\end{array}
\]
Obviously, $f(\lambda, G^{++1})=|M|=|M_1|$, and it is sufficient
to prove our claim for $\lambda\neq 2r-2$. By Lemma \ref{l2},
\begin{eqnarray*}
&&f(\lambda, G^{+++})\\&=& |(\lambda-2r+2)I_m|\cdot
\left|(\lambda-2r)I_n-A+\frac{1}{\lambda-2r+2}R((2r-\lambda-1)R^{\top}+R^{\top}A)\right|\\
&=&(\lambda-2r+2)^{m-n}\\
&&|(\lambda-2r)(\lambda-2r+2)I_n-(\lambda-2r+2)A+(2r-\lambda-1)RR^{\top}+RR^{\top}A|.
\end{eqnarray*}
Let
$B=(\lambda-2r)(\lambda-2r+2)I_n-(\lambda-2r+2)A+(2r-\lambda-1)RR^{\top}+RR^{\top}A.$
Note that  $RR^{\top}=Q$ and $A=Q-rI$. Then
\[
B=(\lambda-2r)(\lambda-2r+2)I_n-(\lambda-2r+2)(Q-rI)+(2r-\lambda-1)Q+Q(Q-rI).\]
 By Lemma \ref{th3}, the eigenvalues of
$B$ are
\[\sigma_n=(\lambda-2r)(\lambda-2r+2)-r(\lambda-2r+2)+2r(2r-\lambda-1)+2r^2=(\lambda-3r+2)(\lambda-4r)\]
and for  $1\leq i\leq n-1$,
\begin{eqnarray*}
\sigma_i&=&(\lambda-2r)(\lambda-2r+2)-(\lambda-2r+2)(q_i-r)+(2r-\lambda-1)q_i+q_i(q_i-r)\\
&=&(\lambda-r-q_i)(\lambda-2r+2-q_i)-q_i.
\end{eqnarray*}
Then
\[|B|=[(\lambda-3r+2)(\lambda-4r)]\prod\limits_{i=1}^{n-1}[(\lambda-r-q_i)(\lambda-2r+2-q_i)-q_i],\]
 and thus the result follows.
\end{proof}

Similarly, we can prove the following theorem.

\begin{theo}
\begin{eqnarray*} f(\lambda,
G^{0++})&=&[(\lambda-r)(\lambda-4r+2)-2r](\lambda-2r+2)^{m-n}\\
&&\prod\limits_{i=1}^{n-1}[(\lambda-r)(\lambda-2r+2-q_i)-q_i],\\
f(\lambda,
G^{1++})&=&[(\lambda-r-2n+2)(\lambda-4r+2)-2r](\lambda-2r+2)^{m-n}
\\&&\cdot
\prod\limits_{i=1}^{n-1}[(\lambda-r-n+2)(\lambda-2r+2-q_i)-q_i],\\
 f(\lambda,
G^{-++})&=&[(\lambda-2n+r+2)(\lambda-4r+2)-2r](\lambda-2r+2)^{m-n}\\&&\prod\limits_{i=1}^{n-1}[(\lambda-n-r+2+q_i)(\lambda-2r+2-q_i)-q_i].
\end{eqnarray*}
\end{theo}

\begin{theo}
\begin{eqnarray*}
f(\lambda,
G^{+-+})&=&[(\lambda-3r)(\lambda-2m+4r-4)-2r](\lambda-m+2r-4)^{m-n}
\\&&
\prod\limits_{i=1}^{n-1}[(\lambda-r-q_i)(\lambda-m+2r-4+q_i)-q_i].
\end{eqnarray*}
\end{theo}
\begin{proof} 
Obviously,
\[
A(G^{+-+})= \left(
\begin{array}{cccc}
A  &  R\\
R^{\top} & J_m-I_m-A(G^l)
\end{array}
\right)=\left(
\begin{array}{cccc}
A  &  R\\
R^{\top} & J_m+I_m-R^{\top}R
\end{array}
\right)\]   and  \[ D(G^{+-+})= \left(
\begin{array}{cccc}
2rI_n  &  0\\
0 & (m-2r+3)I_m
\end{array}
\right).
\]
Then
\begin{eqnarray*}
f(\lambda, G^{+-+})&=& \left|
\begin{array}{cccc}
(\lambda-2r)I_n-A  &  -R\\
     -R^{\top}      & (\lambda-m+2r-4)I_m-J_m+R^{\top}R
\end{array}
\right|.
\end{eqnarray*}
 Let \[M=\left(
\begin{array}{cccc}
(\lambda-2r)I_n-A  &  -R\\
     -R^{\top}      & (\lambda-m+2r-4)I_m-J_m+R^{\top}R
\end{array}
\right).
\]
Multiplying the first row of the block matrix $M$ by
$-\frac{1}{2}J_{mn}+R^{\top}$ and adding the result to the second row of
$M$, we obtain a new matrix
 \begin{eqnarray*}
M_1&=& \left(
\begin{array}{cccc}
(\lambda-2r)I_n -A &  -R\\
(\lambda-2r-1)R^{\top}-R^{\top}A+\frac{3r-\lambda}{2}J_{mn}      &
(\lambda-m+2r-4)I_m
\end{array}
\right).
\end{eqnarray*}
Obviously, $f(\lambda,
G^{+-+})=|M|=|M_1|$, and it is sufficient to prove our claim for
$\lambda\neq m-2r+4$. By Lemma \ref{l2},
\[
f(\lambda, G^{+-+})=(\lambda-m+2r-4)^{m-n}\cdot |B|.
\]
where
$B=(\lambda-2r)(\lambda-m+2r-4)I_n-(\lambda-m+2r-4)A+(\lambda-2r-1)RR^{\top}+\frac{3r-\lambda}{2}RJ_{mn}-RR^{\top}A.$
Note that $RR^{\top}=Q$, $RJ_{mn}=rJ_n$ and $A=Q-rI$. Then
 \begin{eqnarray*}
B&=&(\lambda-2r)(\lambda-m+2r-4)I_n-(\lambda-m+2r-4)(Q-rI)+(\lambda-2r-1)Q\\
&&+\frac{(3r-\lambda)r}{2}J_{n}-Q(Q-rI).
\end{eqnarray*}
By Lemma \ref{th3}, the
eigenvalues of $B$ are
\begin{eqnarray*}
\sigma_n&=&(\lambda-2r)(\lambda-m+2r-4)-r(\lambda-m+2r-4)+2r(\lambda-2r-1)\\
&&+\frac{(3r-\lambda)rn}{2}-2r^2
\\&=&(\lambda-3r)(\lambda-2m+4r-4)-2r
\end{eqnarray*} and for $i=1,2,\dots, n-1$
\begin{eqnarray*}
\sigma_i&=&(\lambda-2r)(\lambda-m+2r-4)-(\lambda-m+2r-4)(q_i-r)\\
&&+(\lambda-2r-1)q_i-q_i(q_i-r)\\
&=&(\lambda-r-q_i)(\lambda-m+2r-4+q_i)-q_i.
\end{eqnarray*}
 Then
\begin{eqnarray*}
|B|&=&[(\lambda-3r)(\lambda-2m+4r-4)-2r]
\\&&\prod\limits_{i=1}^{n-1}[(\lambda-r-q_i)(\lambda-m+2r-4+q_i)-q_i],
\end{eqnarray*}
 and thus the result follows.
\end{proof}

Similarly, we have the following theorem.

\begin{theo}
\begin{eqnarray*}  f(\lambda,
G^{0-+})&=&[(\lambda-r)(\lambda-2m+4r-4)-2r](\lambda-m+2r-4)^{m-n}\\
&&\prod\limits_{i=1}^{n-1}[(\lambda-r)(\lambda-m+2r-4+q_i)-q_i],\\
f(\lambda,
G^{1-+})&=&[(\lambda-r-2n+2)(\lambda-2m+4r-4)-2r](\lambda-m+2r-4)^{m-n}
\\&& \prod\limits_{i=1}^{n-1}[(\lambda-r-n+2)(\lambda-m+2r-4+q_i)-q_i],
\\
 f(\lambda,
G^{--+})&=&[(\lambda-2n+r+2)(\lambda-2m+4r-4)-2r](\lambda-m+2r-4)^{m-n}
\\&&
\prod\limits_{i=1}^{n-1}[(\lambda-n-r+2+q_i)(\lambda-m+2r-4+q_i)-q_i].
\end{eqnarray*}
\end{theo}

\section{Signless Laplacian characteristic polynomials of $G^{xyz}$ with $z=-$}

 Since $G^{00-}$
 is a bipartite graph,  the signless laplacian eigenvalues of $G^{00-}$ are the same as the  laplacian eigenvalues of
 it.
 \[ f(\lambda,
G^{00-})=\lambda(\lambda-n-m+r+2)(\lambda-n+2)^{m-n}\prod
\limits_{i=1}^{n-1}\{(\lambda-m+r)(\lambda-n+2)-q_i\}.\]

\begin{theo}
\begin{eqnarray*}
 f(\lambda,
G^{10-})&=&[(\lambda-n+2)(\lambda-2n+m+r+2)+(2r-m)n-2r]\\&&(\lambda-n+2)^{m-n}
\prod\limits_{i=1}^{n-1}[(\lambda-m-n+r+2)(\lambda-n+2)-q_i].
\end{eqnarray*}
\end{theo}
\begin{proof} Obviously,
 \[
A(G^{10-})= \left(
\begin{array}{cccc}
J_n-I_n  &  J_{nm}-R\\
J_{mn}-R^{\top} & 0
\end{array}
\right) \] and \[ D(G^{10-})= \left(
\begin{array}{cccc}
(n+m-r-1)I_n  &  0\\
0 & (n-2)I_m
\end{array}
\right).
\]
Then
\[
f(\lambda, G^{10-})= \left|
\begin{array}{cccc}
(\lambda-n-m+r+2)I_n-J_n  &  R-J_{nm}\\
    R^{\top}-J_{mn}     & (\lambda-n+2)I_m
\end{array}
\right|.
\]
Clearly, it is sufficient to prove our claim for $\lambda\neq n-2$.
By Lemma \ref{l2},\[f(\lambda, G^{10-})=(\lambda-n+2)^{m-n}|B|,\]
where
\begin{eqnarray*}
B&=&(\lambda-n-m+r+2)(\lambda-n+2)I_n-(\lambda-n+2)J_n\\
&&-(R-J_{nm})(R^{\top}-J_{mn}).
\end{eqnarray*}
Note that  $RR^{\top}=Q$, $J_{nm}R^{\top}=rJ_n$ and $RJ_{mn}=rJ_n$,
then
\[B=(\lambda-n-m+r+2)(\lambda-n+2)I_n+(2r+n-m-2-\lambda)J_n-Q.\]
 By Lemma \ref{th3}, the eigenvalues of $B$ are
\begin{eqnarray*}
\sigma_n&=&(\lambda-n-m+r+2)(\lambda-n+2)+(2r+n-m-2-\lambda)n-2r\\&=&(\lambda-n+2)(\lambda-2n+m+r+2)+(2r-m)n-2r
\end{eqnarray*}
and for $i=1,2,\dots, n-1$,
\[\sigma_i=(\lambda-n-m+r+2)(\lambda-n+2)-q_i.
\]
 Then
\begin{eqnarray*}
|B|&=&[(\lambda-n+2)(\lambda-2n+m+r+2)+(2r-m)n-2r]\\&&\prod\limits_{i=1}^{n-1}
[(\lambda-m-n+r+2)(\lambda-n+2)-q_i],
\end{eqnarray*}
 and thus the result follows.
\end{proof}

Similarly, we have the following theorem.

\begin{theo}
\begin{eqnarray*}f(\lambda,
G^{+0-})&=&[(\lambda-n+2)(\lambda-m-r)+(2r-m)n-2r](\lambda-n+2)^{m-n}\\
&&\prod\limits_{i=1}^{n-1}[(\lambda-n+2)(\lambda-m+r-q_i)-q_i],\\
 f(\lambda,
G^{-0-})&=&[(\lambda-n+2)(\lambda-2n-m+3r+2)+(2r-m)n-2r](\lambda-n+2)^{m-n}\\
&&\prod\limits_{i=1}^{n-1}[(\lambda-n+2)(\lambda-n-m+r+2+q_i)-q_i].
\end{eqnarray*}
\end{theo}

\begin{theo}
\begin{eqnarray*}f(\lambda,
G^{01-})&=&[(\lambda-m+r)(\lambda-n-2m+4)+(4-n)m-2r](\lambda-n-m+4)^{m-n}\\&&\prod\limits_{i=1}^{n-1}
[(\lambda-m-n+4)(\lambda-m+r)-q_i].\end{eqnarray*}
\end{theo}

\begin{proof} Obviously,
\[
A(G^{01-})= \left(
\begin{array}{cccc}
0  &  J_{nm}-R\\
J_{mn}-R^{\top} & J_m-I_m
\end{array}
\right)\]    and \[ D(G^{01-})= \left(
\begin{array}{cccc}
(m-r)I_n  &  0\\
0 & (m+n-3)I_m
\end{array}
\right).
\]
Then
\[
f(\lambda, G^{01-})= \left|
\begin{array}{cccc}
(\lambda-m+r)I_n  &  R-J_{nm}\\
    R^{\top}-J_{mn}     & (\lambda-n-m+4)I_m-J_m
\end{array}
\right|.
\]
Clearly, it is sufficient to prove our claim for $\lambda\neq m-r$.
By Lemma \ref{l2} and the fact  $R^{\top}J_{nm}=2J_n$ and
$J_{mn}R=2J_n$,
\begin{eqnarray*}
&&f(\lambda, G^{01-})\\ &=&
(\lambda-m+r)^{n-m}\\
&&|(\lambda-n-m+4)(\lambda-m+r)I_m-(\lambda-m+r)J_m-(R^{\top}-J_{mn})(R-J_{nm})|
\\&=&
(\lambda-m+r)^{n-m}\\
&&|(\lambda-n-m+4)(\lambda-m+r)I_m-R^{\top}R+(m+4-\lambda-r-n)J_m|
\end{eqnarray*} Let
$B=(\lambda-n-m+4)(\lambda-m+r)I_m-R^{\top}R+(m+4-\lambda-r-n)J_m$.
By Lemma \ref{th3}, the eigenvalues of $B$ are
\begin{eqnarray*}
\sigma_n&=&(\lambda-n-m+4)(\lambda-m+r)-2r+(m+4-\lambda-r-n)m\\&=&(\lambda-m+r)(\lambda-n-2m+4)+(4-n)m-2r,
\end{eqnarray*}
 for $1\leq j\leq m-n$,
 \[
\sigma_i=(\lambda-n-m+4)(\lambda-m+r),\] and for $m-n+1\leq j\leq
m-1$,
\[\sigma_j=(\lambda-n-m+4)(\lambda-m+r)-q_j'=(\lambda-n-m+4)(\lambda-m+r)-q_{j-m+n}.\]
Then
\begin{eqnarray*}|B|&=&[(\lambda-m+r)(\lambda-n-2m+4)+(4-n)m-2r](\lambda-m+r)^{m-n}\\&&(\lambda-n-m+4)^{m-n}\prod\limits_{i=1}^{n-1}
[(\lambda-m-n+4)(\lambda-m+r)-q_i].\end{eqnarray*}
 and thus the result follows.
\end{proof}

Similarly, we have the following theorem.

\begin{theo}
\begin{eqnarray*}
f(\lambda,
G^{0+-})&=&[(\lambda-m+r)(\lambda-n-4r+6)+(4-n)m-2r](\lambda-n-2r+6)^{m-n}\\&&\prod\limits_{i=1}^{n-1}[(\lambda-m+r)(\lambda-n-2r+6-q_i)-q_i],
\\
f(\lambda,
G^{0--})&=&[(\lambda-m+r)(\lambda-n-2m+4r)+(4-n)m-2r](\lambda-n-m+2r)^{m-n}\\&&\prod\limits_{i=1}^{n-1}[(\lambda-m+r)(\lambda-n-m+2r+q_i)-q_i].
\end{eqnarray*}
\end{theo}

\begin{theo}
\begin{eqnarray*}
f(\lambda,
G^{11-})&=&[(\lambda-2n-2m+2)(\lambda-n-m+r+4)+8m](\lambda-n-m+4)^{m-n}\\&&\prod\limits_{i=1}^{n-1}[(\lambda-m-n+r+2)(\lambda-n-m+4)-q_i].
\end{eqnarray*}
\end{theo}

\begin{proof} Obviously,
\[
A(G^{11-})= \left(
\begin{array}{cccc}
J_n-I_n  &  J_{nm}-R\\
J_{mn}-R^{\top} & J_m-I_m
\end{array}
\right) \]  and   \[ D(G^{11-})= \left(
\begin{array}{cccc}
(n+m-r-1)I_n  &  0\\
0 & (m+n-3)I_m
\end{array}
\right).
\]
Then
\begin{eqnarray*}
f(\lambda, G^{11-})&=& \left|
\begin{array}{cccc}
(\lambda-m-n+r+2)I_n-J_n  &  R-J_{nm}\\
    R^{\top}-J_{mn}     & (\lambda-n-m+4)I_m-J_m
\end{array}
\right|\\&=&(2-n)^{-m}| M|, \end{eqnarray*}
where  \[ M= \left(
\begin{array}{cccc}
(\lambda-m-n+r+2)I_n-J_n  &  R-J_{nm}\\
    (2-n)R^{\top}-(2-n)J_{mn}     &(2-n) (\lambda-n-m+4)I_m-(2-n)J_m
\end{array}
\right)\]
Obviously, $J_{mn}R=2J_{mm}$ and $J_{mn}J_{nm}=nJ_{mm}$.
Hence multiplying the first row of the block matrix $M$ by $J_{mn}$
and adding the result to the second row of $M$, we obtain a new
matrix
 \[ M'= \left(
\begin{array}{cccc}
(\lambda-m-n+r+2)I_n-J_n  &  R-J_{nm}\\
    (2-n)R^{\top}+(\lambda-m-n+r)J_{mn}     &(2-n) (\lambda-n-m+4)I_m
\end{array}
\right).\] Clearly, $|M'|=|M|$ and $f(\lambda,
G^{11-})=(2-n)^{-m}|M'|$. Obviously, it is sufficient to prove our
claim for $\lambda\neq n+m-4$. By Lemma \ref{l2}, $f(\lambda,
G^{11-})=(2-n)^{-n}(\lambda-n-m+4)^{m-n}|B|,$ where
\begin{eqnarray*}
B&=&(2-n)(\lambda-m-n+r+2)(\lambda-n-m+4)I_n\\
&&-(2-n)Q-(2-n)(\lambda-n-m+4)J_n\\&&-(\lambda-m+r-2)rJ_n+(\lambda-n-m+r)mJ_n.
\end{eqnarray*}  By Lemma \ref{th3}, the eigenvalues of
$B$ are
\begin{eqnarray*}
\sigma_n&=&(2-n)(\lambda-m-n+r+2)(\lambda-n-m+4)-2r(2-n)\\
&&-(2-n)(\lambda-n-m+4)n\\&&-(\lambda-m+r-2)rn+(\lambda-n-m+r)mn
\\&=&(2-n)[(\lambda-2n-2m+2)(\lambda-n-m+4+r)+8m],
\end{eqnarray*}
 and for $i=1,2,\dots, n-1$
\begin{eqnarray*}
\sigma_i&=&(2-n)(\lambda-m-n+r+2)(\lambda-n-m+4)-(2-n)q_i\\&=&
(2-n)[(\lambda-m-n+r+2)(\lambda-n-m+4)-q_i].
\end{eqnarray*}
 Then
\begin{eqnarray*}
|B|&=&(2-n)^n[(\lambda-2n-2m+2)(\lambda-n-m+r+4)+8m]\\&&\prod\limits_{i=1}^{n-1}
[(\lambda-m-n+r+2)(\lambda-n-m+4)-q_i],
\end{eqnarray*}
 and thus the result follows.
\end{proof}

Similarly, we have the following theorem.

\begin{theo}
\begin{eqnarray*}
 f(\lambda,
G^{+1-})&=&[(\lambda-m-r)(\lambda-n-2m+4)+(4-n)m-2r]\\&&(\lambda-n-m+4)^{m-n}\prod\limits_{i=1}^{n-1}[(\lambda-n-m+4)(\lambda-m+r-q_i)-q_i],
\\
f(\lambda,
G^{-1-})&=&[(\lambda-n-2m+4)(\lambda-2n-m+3r+2)+(4-n)m-2r]\\&&(\lambda-n-m+4)^{m-n}\prod\limits_{i=1}^{n-1}[(\lambda-n-m+4)(\lambda-m-n+r+q_i+2)-q_i].
\end{eqnarray*}
\end{theo}

\begin{theo}
\begin{eqnarray*}
f(\lambda,
G^{1+-})&=&[(\lambda-2n-m+r+2)(\lambda-n-4r+6)+(4-n)m-2r]\\&&(\lambda-n-2r+6)^{m-n}\prod\limits_{i=1}^{n-1}[(\lambda-n-m+r+2)(\lambda-n-2r+6-q_i)-q_i].
\end{eqnarray*}
\end{theo}
\begin{proof} Obviously,
\[
A(G^{1+-})= \left(
\begin{array}{cccc}
J_n-I_n &  J_{nm}-R\\
J_{mn}-R^{\top} & A(G^l)
\end{array}
\right)=\left(
\begin{array}{cccc}
J_n-I_n &  J_{nm}-R\\
J_{mn}-R^{\top} & R^{\top}R-2I_m
\end{array}
\right)
 \]  and   \[ D(G^{1+-})= \left(
\begin{array}{cccc}
(n+m-r-1)I_n  &  0\\
0 & (n+2r-4)I_m
\end{array}
\right).
\]
Then
\begin{eqnarray*}
f(\lambda, G^{1+-})&=& \left|
\begin{array}{cccc}
(\lambda-n-m+r+2)I_n-J_n  &  R-J_{nm}\\
    R^{\top}-J_{mn}     & (\lambda-n-2r+6)I_m-R^{\top}R
\end{array}
\right|. \end{eqnarray*}

Let  \begin{eqnarray*} M= \left(
\begin{array}{cccc}
(\lambda-n-m+r+2)I_n-J_n  &  R-J_{nm}\\
    R^{\top}-J_{mn}     & (\lambda-n-2r+6)I_m-R^{\top}R
\end{array}
\right), \end{eqnarray*} then $f(\lambda, G^{1+-})=|M|$. Obviously,
$R^{\top}J_{nm}=2J_m$. Thus multiplying the first row of the block
matrix $M$ by $R^{\top}$ and adding the result to the second row of
$M$, we obtain a new matrix
\begin{eqnarray*}
  M'= \left(
\begin{array}{cccc}
(\lambda-n-m+r+2)I_n-J_n  &  R-J_{nm}\\
   (\lambda-n-m+r+3) R^{\top}-3J_{mn}     & (\lambda-n-2r+6)I_m-2J_m
\end{array}
\right).\end{eqnarray*}
 Let
 \begin{eqnarray*}
 M''= \left(
\begin{array}{cccc}
(\lambda-n-m+r+2)I_n-J_n  &  R-J_{nm}\\
  (2-n)(\lambda-n-m+r+3) R^{\top}\\-3(2-n)J_{mn}     &
  (2-n)(\lambda-n-2r+6)I_m-2(2-n)J_m
\end{array}
\right).\end{eqnarray*} Clearly, $|M''|=(2-n)^{m}|M'|=(2-n)^{m}|M|$
and $f(\lambda, G^{11-})=|M'|=(2-n)^{-m}|M''|$. Obviously,
$J_{mn}R=2J_m$ and $J_{mn}J_{nm}=nJ_m$. Thus multiplying the first
row of the block matrix $M''$ by $2J_{mn}$ and adding the result to
the second row of $M''$, we obtain a new matrix
 \begin{eqnarray*}
  M'''= \left(
\begin{array}{cccc}
(\lambda-n-m+r+2)I_n-J_n  &  R-J_{nm}\\
   (2-n)(\lambda-n-m+r+3) R^{\top}\\+(2\lambda-n-2m+2r-2)J_{mn}     & (2-n)(\lambda-n-2r+6)I_m
\end{array}
\right).\end{eqnarray*}
 Clearly, $|M'''|=|M''|$, $f(\lambda,
G^{11-})=(2-n)^{-m}|M'''|$, and it is sufficient to prove our claim
for $\lambda\neq n+2r-6$. By Lemma \ref{l2},
\[|M'''|=(2-n)^{m-n}(\lambda-n-2r+6)^{m-n}|B|,\]
 where
\begin{eqnarray*}
B&=&(2-n)(\lambda-n-2r+6)(\lambda-n-m+r+2)I_n-(2-n)(\lambda-n-2r+6)J_n\\
&&-(R-J_{nm})[(2-n)(\lambda-n-m+r+3)
R^{\top}+(2\lambda-n-2m+2r-2)J_{mn}]\\
&=&(2-n)(\lambda-n-2r+6)(\lambda-n-m+r+2)I_n\\
&&-(2-n)(\lambda-n-m+r+3)RR^{\top}\\
&&-(2-n)(\lambda-n-2r+6)J_n-
(4n-8+\lambda n-n^2-nm+nr)rJ_n\\
&&+(2\lambda-n-2m+2r-2)mJ_n.
\end{eqnarray*}
Thus \[f(\lambda, G^{11-})=(2-n)^{-n}(\lambda-n-2r+6)^{m-n}|B|.\] By
Lemma \ref{th3}, the eigenvalues of $B$ are
\begin{eqnarray*}
\sigma_n&=&(2-n)(\lambda-n-2r+6)(\lambda-n-m+r+2)-(2-n)(\lambda-n-m+r+3)\cdot
2r\\&&-(2-n)(\lambda-n-2r+6)n- (4n-8+\lambda
n-n^2-nm+nr)rn\\&+&(2\lambda-n-2m+2r-2)mn
\\&=&(2-n)[(\lambda-2n-m+r+2)(\lambda-n-4r+6)+(4-n)m-2r],
\end{eqnarray*}
 and for $i=1,2,\dots, n-1$,
\begin{eqnarray*}
\sigma_i&=&(2-n)(\lambda-n-2r+6)(\lambda-n-m+r+2)-(2-n)(\lambda-n-m+r+3)q_i\\&=&
(2-n)[(\lambda-n-m+r+2)(\lambda-n-2r+6-q_i)-q_i].
\end{eqnarray*}
 Then  \begin{eqnarray*} |B|&=&(2-n)^n[(\lambda-2n-m+r+2)(\lambda-n-4r+6)+(4-n)m-2r]\\&&\prod\limits_{i=1}^{n-1}
[(\lambda-n-m+r+2)(\lambda-n-2r+6-q_i)-q_i],
\end{eqnarray*}
 and thus the result follows.
\end{proof}

Similarly, we have the following theorem.

\begin{theo}
\begin{eqnarray*}
 f(\lambda,
G^{++-})&=&[(\lambda-m-r)(\lambda-n-4r+6)+(4-n)m-2r](\lambda-n-2r+6)^{m-n}\\&&\prod\limits_{i=1}^{n-1}[(\lambda-m+r-q_i)(\lambda-n-2r+6-q_i)-q_i],
\\
f(\lambda,
G^{-+-})&=&[(\lambda-n-4r+6)(\lambda-2n-m+3r+2)+(4-n)m-2r]\\&&(\lambda-n-2r+6)^{m-n}\\
&&\prod\limits_{i=1}^{n-1}[(\lambda-m-n+r+2+q_i)(\lambda-n-2r+6-q_i)-q_i].
\end{eqnarray*}
\end{theo}

\begin{theo}
\begin{eqnarray*}
f(\lambda,
G^{1--})&=&[(\lambda-n-2m+4r)(\lambda-2n-m+r+2)+(4-n)m-2r]\\&&(\lambda-n-m+2r)^{m-n}\\
&&\prod\limits_{i=1}^{n-1}[(\lambda-n-m+r+2)(\lambda-n-m+2r+q_i)-q_i].
\end{eqnarray*}
\end{theo}
\begin{proof} Obviously,
\[
A(G^{1--})= \left(
\begin{array}{cccc}
J_n-I_n &  J_{nm}-R\\
J_{mn}-R^{\top} & J_m-I_m-A(G^l)
\end{array}
\right)=\left(
\begin{array}{cccc}
J_n-I_n &  J_{nm}-R\\
J_{mn}-R^{\top} & J_m+I_m-R^{\top}R
\end{array}
\right)
 \]  and   \[ D(G^{1--})= \left(
\begin{array}{cccc}
(n+m-r-1)I_n  &  0\\
0 & (n+m-2r-1)I_m
\end{array}
\right).
\]
Then
\begin{eqnarray*}
f(\lambda, G^{1--})&=& \left|
\begin{array}{cccc}
(\lambda-n-m+r+2)I_n-J_n  &  R-J_{nm}\\
    R^{\top}-J_{mn}     & (\lambda-n-m+2r)I_m-J_m+R^{\top}R
\end{array}
\right|. \end{eqnarray*} Let \begin{eqnarray*}
 M= \left(
\begin{array}{cccc}
(\lambda-n-m+r+2)I_n-J_n  &  R-J_{nm}\\
    R^{\top}-J_{mn}     & (\lambda-n-m+2r)I_m-J_m+R^{\top}R
\end{array}
\right). \end{eqnarray*} Then $f(\lambda, G^{1--})=|M|$. Note that
$R^{\top}J_{nm}=2J_m$ . Hence multiplying the first row of the block
matrix $M$ by $-R^{\top}$ and adding the result to the second row of
$M$, we obtain a new matrix
 \begin{eqnarray*}
  M_1= \left(
\begin{array}{cccc}
(\lambda-n-m+r+2)I_n-J_n  &  R-J_{nm}\\
   (n+m-\lambda-r-1) R^{\top}+J_{mn}     & (\lambda-n-m+2r)I_m+J_m
\end{array}
\right).
\end{eqnarray*}
Let \[
 M_2= \left(
\begin{array}{cccc}
(\lambda-n-m+r+2)I_n-J_n  &  R-J_{nm}\\
   (2-n)(n+m-\lambda-r-1) R^{\top}\\+(2-n)J_{mn}
   &(2-n)(\lambda-n-m+2r)I_m+(2-n)J_m
\end{array}
\right).\] Clearly, $|M_2|=(2-n)^{-m}|M_1|=(2-n)^{-m}|M|$ and
$f(\lambda, G^{1--})=|M_1|=(2-n)^{-m}|M_2|$. Obviously,
$J_{mn}R=2J_m$ and $J_{mn}J_{nm}=nJ_m$. Thus multiplying the first
row of the block matrix $M''$ by $-J_{mn}$ and adding the result to
the second row of $M_2$, we obtain a new matrix
 \[
  M_3= \left(
\begin{array}{cccc}
(\lambda-n-m+r+2)I_n-J_n  &  R-J_{nm}\\
   (2-n)(n+m-\lambda-r-1) R^{\top}+(n+m-\lambda-r)J_{mn}     &(2-n)(\lambda-n-m+2r)I_m
\end{array}
\right).\] Clearly, $|M_3|=|M_2|$,  $f(\lambda,
G^{1--})=(2-n)^{-m}|M_3|$, and it is sufficient to prove our claim
for $\lambda\neq n+m-2r$. By Lemma \ref{l2},
\[|M_3|=(2-n)^{m-n}(\lambda-n-m+2r)^{m-n}|B|,\]where
\begin{eqnarray*}
B&=&(2-n)(\lambda-n-m+2r)(\lambda-n-m+r+2)I_n-(2-n)(\lambda-n-m+2r)J_n
\\&&-(R-J_{nm})((2-n)(n+m-\lambda-r-1) R^{\top}+(n+m-\lambda-r)J_{mn})\\&=&(2-n)(\lambda-n-m+2r)(\lambda-n-m+r+2)I_n-(2-n)(n+m-\lambda-r-1)Q
\\&&-(2-n)(\lambda-n-m+2r)J_n+(m+n\lambda+nr+2n-\lambda-r-n^2-nm-2)rJ_n\\&&+(n+m-\lambda-r)mJ_n.
\end{eqnarray*}
Thus  \[f(\lambda, G^{1--})=(2-n)^{-n}(\lambda-n-m+2r)^{m-n}|B|.\]
 By Lemma \ref{th3}, the eigenvalues of
$B$ are
\begin{eqnarray*}
\sigma_n&=&(2-n)(\lambda-n-m+2r)(\lambda-n-m+r+2)-(2-n)(n+m-\lambda-r-1)2r\\&&-(2-n)(\lambda-n-m+2r)n
+(m+n\lambda+nr+2n-\lambda-r-n^2-nm-2)rn\\&&+(n+m-\lambda-r)mn
\\&=&(2-n)[(\lambda-n-2m+4r)(\lambda-2n-m+r+2)+(4-n)m-2r],
\end{eqnarray*}
 and
\begin{eqnarray*}
\sigma_i&=&(2-n)(\lambda-n-m+2r)(\lambda-n-m+r+2)-(2-n)(n+m-\lambda-r-1)q_i\\&=&
(2-n)[(\lambda-n-m+r+2)(\lambda-n-m+2r+q_i)-q_i] \end{eqnarray*}
for $i=1,2,\dots, n-1$.
Then
 \begin{eqnarray*}
|B|&=&(2-n)^n[(\lambda-n-2m+4r)(\lambda-2n-m+r+2)+(4-n)m-2r]\\
&&\prod\limits_{i=1}^{n-1}[(\lambda-n-m+r+2)(\lambda-n-m+2r+q_i)-q_i],
\end{eqnarray*}
 and thus the result follows.
 \end{proof}

Similarly, we have the following theorem.

\begin{theo}
\begin{eqnarray*}
f(\lambda,
G^{+--})&=&[(\lambda-m-r)(\lambda-n-2m+4r)+(4-n)m-2r]\\&&(\lambda-n-m+2r)^{m-n}\\
&&\prod\limits_{i=1}^{n-1}[(\lambda-n-m+2r+q_i)(\lambda+r-m-q_i)-q_i].
\end{eqnarray*}
\end{theo}


\begin{theo}
\begin{eqnarray*}
f(\lambda,
G^{---})&=&(\lambda-2n-2m+4r+2)(\lambda+3r-n-m)(\lambda+2r-n-m)^{m-n}\\&&\prod\limits_{i=1}^{n-1}[(\lambda-n-m+r+q_i+2)(\lambda+2r-n-m+q_i)-q_i].
\end{eqnarray*}
\end{theo}

\begin{proof} Note that  $G^{---}$ is the complement of  $G^{+++}$, and $G^{+++}$ is $2r$-regular.
By  Lemma \ref{th2} and Theorem \ref{th4},
\begin{eqnarray*}
&&f(\lambda,
G^{---})\\
&=&(-1)^{m+n-1}\frac{\lambda-2m-2n+2+4r}{n+m-2-\lambda-4r}(n+m-2-\lambda-4r)(n+m-\lambda-3r)\\&&(m+n-\lambda-2r)^{m-n}
\prod\limits_{i=1}^{n-1}[(n+m-\lambda-r-q_i-2)(n+m-\lambda-2r-q_i)-q_i]\\
&=&(\lambda-2n-2m+4r+2)(\lambda+3r-n-m)(\lambda+2r-n-m)^{m-n}\\&&\prod\limits_{i=1}^{n-1}[(\lambda-n-m+r+q_i+2)(\lambda+2r-n-m+q_i)-q_i],
\end{eqnarray*}
as desired.
\end{proof}

\vspace{4mm}

\noindent {\bf Acknowledgement.} This work was supported by the Research Fund for the Doctoral Program of Higher Education of China (No.~20124407110002).


\begin{thebibliography}{9}


\bibitem{DM} D.M. Cvetkovi\'c,  M. Doob, H. Sachs, Spectra of Graphs: Theory and
Applications, third ed., Johann Ambrosius Barth Verlag, Heidelberg,
Leipzig, 1995.

\bibitem{AK} A.K. Kelmans, The properties of the characteristic polynomial of a graph, in: Cybernetics¡ªin the Service of Communism 4 (Russian), Izdat. ``\`Energija'', Moscow, 1967, pp. 27--41.





\bibitem{JK} J. Yan, K. Xu, Spectra of transformation graphs of regular graph,
Appl. Math. J. Chinese Univ. Ser. A 23 (2008) 476-480.

\bibitem{AAJ} A. Deng, A. Kelmans, J. Meng, Laplacian spectra of regular graph transformation, Discrete Appl. Math. 161 (2013) 118--133.





%
\end{thebibliography}
\end{document}